\documentclass[11pt,letterpaper]{amsart}
\usepackage{amssymb}
\usepackage{amsfonts}
\usepackage{amssymb}
\usepackage{amsmath}
\usepackage{amsthm}
\usepackage{enumerate}
\usepackage{amsmath}
\usepackage{tabularx}
\usepackage{centernot}
\usepackage{mathtools}
\usepackage{stmaryrd}
\usepackage{amsthm,amssymb}
\usepackage{etoolbox,color}
\usepackage{tikz}
\usepackage{amssymb}
\usetikzlibrary{matrix}
\usepackage{tikz-cd}
\usepackage{tikz}
\usetikzlibrary{calc}
 \usepackage[square, numbers]{natbib}
\newlength\mylen
\setcitestyle{citesep={,\kern-\mylen}}
\usepackage{marginnote}
\definecolor{mygray}{gray}{0.85}
\usepackage[linecolor=black, backgroundcolor=mygray,colorinlistoftodos,prependcaption,textsize=small]{todonotes}
\usepackage{xargs}  
\usepackage[colorlinks,citecolor=blue,urlcolor=blue, linkcolor=blue]{hyperref}
\newcommand{\ol}{\overline}
\newcommand{\ul}{\underline}
\newcommand\+[1]{\mathcal{#1}}
\newcommand{\NN}{\mathbb N}

\renewcommand{\leq}{\leqslant}

\renewcommand{\geq}{\geqslant}

\newcommand{\mrm}[1]{\mathrm{#1}}

\usepackage[backgroundcolor=mygray,colorinlistoftodos,prependcaption,textsize=small]{todonotes}
\usepackage{xcolor}
  \newcommand{\pp}[3]{\prod_{#2 \in #3}   {x_#2}^{{#1}_#2}} 
 \newcommand{\sub} {\subseteq}

\newcommand{\uhr}{\upharpoonright}

\renewcommand{\S}{\mathrm{Sym}(\mathbb N)}

\makeatletter
\def\subsection{\@startsection{subsection}{3}%
	\z@{.5\linespacing\@plus.7\linespacing}{.3\linespacing}%
	{\bfseries\centering}}
\makeatother

\makeatletter
\def\subsubsection{\@startsection{subsubsection}{3}%
	\z@{.5\linespacing\@plus.7\linespacing}{.3\linespacing}%
	{\centering}}
\makeatother

\makeatletter
\def\myfnt{\ifx\protect\@typeset@protect\expandafter\footnote\else\expandafter\@gobble\fi}
\makeatother

\newcommand{\Gps}{\mathrm{GrpEpi}}

\newtheorem{theorem}{Theorem}[section]

\newtheorem{corollary}[theorem]{Corollary}

\newtheorem{lemma}[theorem]{Lemma}
\newtheorem{proposition}[theorem]{Proposition}

\newtheorem{question}[theorem]{Question}

\theoremstyle{plain}

\theoremstyle{definition}

\newtheorem{fact}[theorem]{Fact}
\newtheorem{definition}[theorem]{Definition}
\newtheorem{remark}[theorem]{Remark}

\newtheorem{notation}[theorem]{Notation}

\begin{document}

\begin{abstract} 
We prove that topological isomorphism on  procountable groups is not classifiable by countable structures, in the sense of descriptive set theory. In fact, the equivalence relation $\ell_\infty$    expressing  that two sequences of reals have a bounded difference is Borel reducible to it.  This marks substantial progress on an open problem of Kechris, Nies and Tent (2018):  to determine the exact  complexity of  the isomorphism  relation among all non-archimedean Polish groups.
\end{abstract}


\thanks{S.\ Gao acknowledges the partial support of their research by the Fundamental Research Funds for the Central Universities and by the National Natural Science Foundation of China (NSFC) grant 12271263. A.\ Nies acknowledges support by the Marsden Fund of New Zealand, UOA-346. Both S.\ Gao and A.\ Nies attended the Tianyuan Workshop on Computability Theory and Descriptive Set Theory in June 2025 when a part of this research was conducted, and we acknowledge the support by the Tianyuan Mathematics Research Center. 
G.\ Paolini was supported by project PRIN 2022 ``Models, sets and classifications", prot. 2022TECZJA and by INdAM Project 2024 (consolidator grant) ``Groups, Crystals and Classifications''.}

\title[Procountable groups are not CCS]{Procountable groups are  not classifiable   \\ by countable structures}

\author{Su Gao}
\address{School of Mathematical Sciences and LPMC, Nankai University, Tianjin 300071, P.R. China}
\email{sgao@nankai.edu.cn}

\author{Andr{\'e} Nies}
\address{School of Computer Science, The University of Auckland, New Zealand}
\email{andre@cs.auckland.ac.nz}

\author{Gianluca Paolini}
\address{Department of Mathematics ``Giuseppe Peano'', University of Torino, Via Carlo Alberto 10, 10123, Italy.}
\email{gianluca.paolini@unito.it}

\maketitle


\section{Introduction} It is often useful  to  have concrete presentations of  abstract   mathematical objects of a certain type. In this case one is interested in  the complexity of  whether two presentations present the same abstract object.   For instance,   finite graphs can be concretely presented by   adjacency matrices (written as binary strings).  The graph isomorphism problem asks whether two such  matrices   present the same graph. This   problem is  in NP, and its exact  complexity    has been the object of much recent study  \cite{Babai2016,Wiebking2019}. 
For many types of mathematical objects,  a  countably infinite amount of information is required to present a member.  In this case,  one can usually view the presentations as points in a so-called standard Borel space (such as the unit interval). Borel reducibility, introduced in \cite{FriedmanStanley1989}, is the standard tool for  comparing the  complexity of  equivalence relations on standard Borel spaces.  For instance, 
 countable graphs can be  concretely presented as symmetric, irreflexive relations on $\mathbb N$, and their isomorphism is a well-known benchmark equivalence relation~\cite{GaoBook}   denoted by GI.  The isomorphism relation for a class of concretely presented objects  is   reducible to GI  if and only there is a Borel assignment of  countable structures as complete invariants to      presentations of the  objects in question.


 \medskip 
 
 \noindent \textit{The main result.} A Polish group is a topological group with a Polish topology, namely, one that is  separable and completely metrizable. A topological  group is  called non-archimedean if it has a neighbourhood basis of the identity consisting of open subgroups \cite{BeckerKechris1996}. Such groups are totally disconnected. Let $\S$ denote the group of all permutations of $\mathbb N$ equipped with the pointwise convergence topology.  Then $\S$ is a non-archimedean Polish group. Up to topological isomorphism, the non-archimedean Polish groups  are precisely the closed subgroups of $\S$ \cite{BeckerKechris1996}.
 It is well known that the closed subgroups of  $\S$ form a standard Borel space; see for instance  Kechris et al. \cite{KNT}.  So the closed subgroups of $\S$ are concrete presentation of non-archimedean Polish groups.

 In the same article~\cite{KNT},  the authors initiated the following programme: 
\begin{quote}
Given a  Borel class $\+ C$ of closed subgroups of $\S$,     determine the     complexity  of  topological isomorphism between  groups in $\+ C$.  
\end{quote}
 Our work marks considerable progress on an open problem of Kechris et al.\ \cite{KNT} pertaining to their programme. 
  They  asked at the end of their introduction for  the complexity of isomorphism on    the class of \textit{all} non-archimedean Polish groups, observing that this relation is  analytic, and  GI is Borel reducible to it.  By our  main result, Theorem~\ref{th:main}  below, 
 the isomorphism relation on the Borel class of  procountable groups   already has a complexity much higher than~GI. 
  \begin{definition} \label{def:procountable}
	 A topological group   $G$ is called  \textit{procountable} if $G$ is topologically isomorphic to  the    inverse limit of  an inverse system $(G_n, p_n)_{n \in \mathbb N} $ of countable groups and onto homomorphisms $p_n \colon G_{n+1} \to G_n$.  The topology of $G$ is the subspace topology given by the product $\prod_n G_n$, where each $G_n$ is equipped with the discrete topology.  
\end{definition}
Procountable groups are of  interest because they generalise both profinite and discrete groups, and on the other hand they are precisely the non-archimedean   groups that have  small invariant neighbourhoods (SIN). See \cite{Pestov:99} for a discussion of  this important class of topological groups. It is easily seen that a group of the form $\prod_n G_n$, and hence each procountable group, is  isomorphic to a closed subgroup of $\S$. 
For background and further equivalent characterisations, see e.g., \cite{Malicki2016}.  We will verify in  \ref{prop: Borel pc}  that   the procountable closed subgroups of $\S$ form a  Borel set, using that the procountable groups  are   the non-archimedean groups with  a neighbourhood basis of the neutral element consisting of \textit{normal} open subgroups~(see, e.g., \cite[Theorem 3.3]{GX} or \cite[Lemma~2]{Malicki2016}). 

   Since each countable   discrete group is procountable,    by a  result of Mekler~\cite{Mekler1985}, GI is reducible to the isomorphism relation on 2-nilpotent procountable groups of exponent~$p$. 
It follows that the topological isomorphism relation of procountable groups is a complete analytic set and hence non-Borel.  

   Recall that $E_1$ is the Borel equivalence relation of eventual equality of  sequences of real numbers. By a celebrated theorem of Kechris and Louveau~\cite{KL}, $E_1$ is not Borel reducible to any orbit equivalence relation induced by a Borel action of a Polish group. It follows that any equivalence relation to which  $E_1$ is  Borel reducible has the same property. 
Recall that   $\ell_\infty$ is the Borel  equivalence relation on $\mathbb{R}^\mathbb{N}$ given by 
$$ (x_n)\ell_\infty(y_n)\iff \exists  M \in \mathbb R \ \forall n\in\mathbb{N}\ |x_n-y_n|<M. $$  
By  a result of Rosendal (see \cite[Theorem 8.4.2]{GaoBook}), $E_1$ is Borel reducible to $\ell_\infty$, and thus $\ell_\infty$ is not Borel reducible to the  orbit equivalence relation of any  Borel  action of a Polish group. The following is the main result of this paper.

 \begin{theorem}\label{th:main}
 Fix an odd prime $p$.	The equivalence relation $\ell_\infty$ is Borel reducible to  the topological isomorphism relation on  procountable groups. Consequently, the topological isomorphism relation on   procountable groups is not Borel reducible to any orbit equivalence relation induced by a Borel action of a Polish group. In particular, it is not Borel reducible to GI.
 \end{theorem}
 
It is known that being Borel reducible to GI is equivalent to being below the orbit equivalence relation of a Borel  action of $\S$, which is also equivalent to being classifiable by some sort of countable structures. It follows from our main theorem that the topological isomorphism of procountable groups  is not classifiable by countable structures.  

Kechris et al.\ \cite[Theorem 3.1]{KNT} provided a criterion when the topological isomorphism relation on a class  $\+ C$  is Borel reducible to  GI; intuitively, it says that one can canonically and in a Borel way assign to $G \in \+ C$ a countable neighbourhood basis of the identity consisting of open subgroups.  By our main theorem, the hypothesis of this criterion is violated in our setting: 
\begin{corollary}
{\rm Let $\mathcal{C}$ be the Borel class of all  procountable groups. There is no    assignment $G\mapsto \mathcal{N}_G$ such that  for each $G\in \mathcal{C}$, $\mathcal{N}_G$ is a countable set of open subgroups of $G$ that form a neighborhood basis of $1_G$, and this assignment is Borel and isomorphism invariant in the sense of \cite[Theorem~3.1]{KNT}. }
\end{corollary}

%
%
%

The proof of our main theorem makes use of the notion of uniform homeomorphism.
		\begin{definition} \label{def:UH} 
	A \emph{Polish metric space} is a metric space $(X, d)$ where $d$ is a complete metric and the topology generated by $d$ is Polish. Let $X$ and $Y$ be Polish metric spaces. A bijection  $\Phi$ between  $X$ and $Y$ is called a \emph{uniform homeomorphism} if   both $\Phi$ and its inverse are uniformly continuous.  $X$ and $Y$ are said to be \emph{uniformly homeomorphic} if there is a uniform homeomorphism between $X$ and $Y$. 
	\end{definition}

	\noindent \textit{The two steps of the proof  of the main result \ref{th:main}}.  
\begin{enumerate} \item[1.]  Section~\ref{s:1}   shows that $\ell_\infty$ is Borel reducible to the relation of uniform homeomorphism between path spaces $[T]$,  where $T$ is a pruned subtree of $\mathbb{N}^{< \mathbb{N}}$ and $[T]$ carries the standard ultrametric. \item[2.]  Section~\ref{s:2}  reduces this uniform homeomorphism relation to the topological isomorphism relation on the Borel space of procountable   groups.
\end{enumerate} 
 
  The second step only uses a very particular type of groups that could be called  ``solenoid groups". All the  groups  $G_n$ in the inverse system $(G_n, p_n)_{n \in \mathbb N} $ are  isomorphic.  Each group $G_n$ is generated by a set $V_n$, and   the binding maps $p_n$ are determined  by surjections $V_{n+1} \to V_n$.

%

\medskip 
\noindent {\it Some   background on the program of \cite{KNT}.}
For context  we summarize other  results    known to us that pertain to the program.
One class  considered in \cite{KNT} is  the class of compact non-archimedean Polish groups, or equivalently, the profinite groups. It was  shown that its   isomorphism relation is Borel bi-reducible with~GI. This implies that this isomorphism relation is a complete analytic set, and hence non-Borel.  
One can also consider the subclass of abelian profinite groups. Due to   Pontryagin duality, the isomorphism relation between abelian profinite groups is Borel bi-reducible with
the isomorphism relation between countable torsion abelian groups. The  latter relation is classified by   Ulm invariants; it is also   a complete analytic set and hence non-Borel (see \cite[(2.1)]{FriedmanStanley1989}). 
In terms of Borel reducibility, it is strictly above the smooth equivalence relation and strictly below~GI. In particular,
it is not above the equivalence relation~$E_0$.
The isomorphism relation of  oligomorphic groups was shown to be essentially countable in \cite{NiesSchlichtTent2022};  it is at present an open question whether this relation is   smooth (i.e.,    real numbers can be assigned as invariants); for partial results on this problem,  see \cite{NP_oligo, reconstruction2}; in particular in \cite{reconstruction2} it was shown that under the assumption of weak elimination of imaginaries the relation is smooth.
The isomorphism relation of extremely amenable non-archimedean Polish groups was studied in \cite{EGL}, where it is shown that GI is Borel reducible to it.  

We note that the programme of \cite{KNT}     asked to first determine whether a given class of interest~$\+ C$ of closed subgroups of $\S$   is indeed Borel. While this is often straightforward from the definition, in some cases it   takes work to establish.    Borelness is unknown for the class of topologically finitely generated groups, and known to fail   for some  classes of interest. For instance,  the class of all CLI groups (those admitting a compatible complete left-invariant metric) was shown to be properly co-analytic   by   Malicki~\cite{Ma}.   Paolini and Shelah~\cite{PS_polish} proved  that four  properties  of closed abelian subgroups of $\S$ are  complete co-analytic and therefore non-Borel: separability (in the sense of abelian group theory), torsionlessness, $\aleph_1$-freeness and $\mathbb{Z}$-homogeneity. 

Countable groups are up to isomorphism the  discrete closed subgroups $G$ of $\S$, and being discrete is Borel because it means that  the identity permutation is isolated in~$G$. If a class of discrete subgroups of $\S$  is  Borel, its   isomorphism relation is Borel reducible to GI. The question   is whether conversely,  GI   reduces to the  isomorphism relation on the class. 
Paolini and Shelah~\cite{PS}
showed it for  countable torsion-free abelian groups.  Most recently, Gao and Li~\cite{GL}
did so for countable omnigenous locally finite groups (an uncountable class of groups, each generalising  Hall's universal locally finite group).

\medskip

\noindent {\it Open questions.} An analytic equivalence relation on a Borel space is called  universal  analytic if each analytic equivalence relation is Borel reducible to~it.

	\begin{question}\label{universal_question} Is the equivalence  relation of topological isomorphism on procountable groups  universal analytic? 
\end{question}
\noindent 

It is known that    the relation of uniform homeomorphism on Polish metric spaces is universal analytic~\cite{FLR}.
\begin{question}\label{universal_question_+} Is the  equivalence relation of uniform homeomorphism on Polish \emph{ultra}metric  spaces  universal analytic?
\end{question}
\noindent  By the second step above, an affirmative answer to the second  question would also answer the first   in the affirmative.    The preceding two  questions will be discussed again in see Section~\ref{Sec_fr}.

	\begin{question} Are the abelian procountable groups classifiable by countable structures?
\end{question}
\noindent
Notice that it follows from \cite{PS} that graph isomorphism is a lower bound for the complexity of  this classification problem, so an affirmative answer  would fully determine the complexity.

\section{Borel reduction of $\ell_\infty$ to uniform homeomorphism} \label{s:1}

\begin{definition} \label{def:tree} 
	By a    \emph{tree on $\mathbb{N}$} we mean a  nonempty set $T\subseteq \mathbb{N}^{<\mathbb{N}}$ that is closed under initial segments. The path space 
	$$ [T]=\big\{ x\in \mathbb{N}^\mathbb{N}\colon \forall n\in\mathbb{N}\ (x\!\!\upharpoonright\! n\in T)\big\} $$
	is a Polish ultrametric space with the metric given     by declaring for  $x\neq y$ that 
	$$ d(x,y)=2^{-n}, \mbox{ where $n\in\mathbb{N}$ is the least such that $x(n)\neq y(n)$.} $$
	The elements of $[T]$ will be called  {\em branches} of $T$.
\end{definition}
 A tree   $T$ is called   \emph{pruned} if for any $s\in T$ there is $t\in T$ such that $s\subsetneq t$.  The space of   all pruned trees on $\mathbb{N}$  can be seen to be a standard Borel (even Polish) subspace of $2^{(\mathbb{N}^{<\mathbb{N}})}$, which will be denoted~$\mathcal{T}$.  
 
 Note that  the paths ultrametric spaces are in a canonical bijection with the pruned trees, so we can see the elements of $\mathcal T$ as codes for such spaces. 
 
 \begin{definition} \label{def:cong u}
 	By  $\cong_u$ we   denote the equivalence relation of  uniform  homeomorphism on $\mathcal{T}$.
 \end{definition}

The following quick observation   gives a lower bound on the complexity of $\cong_u$. 
\begin{proposition} \label{prop:quicko}
	$\mbox{\rm GI}$ is Borel reducible to $\cong_u$.
\end{proposition}
\begin{proof}
	We use the main result of  \cite{CG},  where it was shown that GI is Borel bireducible with the homeomorphism relation between $0$-dimensional compact metric spaces. Let $\Phi$ denote a Borel witness to this reduction. Then for any graphs $\Gamma, \Gamma'$ with vertex set $\mathbb{N}$, $\Gamma$ and $\Gamma'$ are isomorphic if and only if $\Phi(\Gamma)$ and $\Phi(\Gamma')$ are homeomorphic. Note that a Polish space is $0$-dimensional if and only if it admits a compatible ultrametric. So there exists a compatible ultrametric $d_\Gamma$ for $\Phi(\Gamma)$. By the standard tools from descriptive set theory, we may find a Borel map giving the correspondence $\Phi(\Gamma)\mapsto d_\Gamma$. 
	
\smallskip \noindent Note that for any compact ultrametric space $(X, d_X)$, the set of nonzero distances $\{d_X(x,y)\colon x,y\in X\}\setminus\{0\}$ is a decreasing sequence of positive real numbers converging to $0$. Up to homeomorphism, we may assume that this sequence is $\{2^{-n}\colon n\in\mathbb{N}\}$. Now that $(\Phi(\Gamma), d_\Gamma)$ is a compact ultrametric space with its set of nonzero distances a subset of $\{2^{-n}\colon n\in\mathbb{N}\}$, there is a unique tree $T_\Gamma\in\mathcal{T}$ such that $(\Phi(\Gamma), d_\Gamma)$ is isometric to $[T_\Gamma]$. Again the map $d_\Gamma\mapsto T_\Gamma$ is Borel.

\smallskip \noindent Finally, note that, as compact spaces, $[T_{\Gamma}]$ and $[T_{\Gamma'}]$ are homeomorphic if and only if they are uniformly homeomorphic.  Thus we have that $\Gamma$ and $\Gamma'$ are isomorphic if and only if $[T_\Gamma]$ and $[T_{\Gamma'}]$ are uniformly homeomorphic as compact ultrametric spaces. This proves the proposition. 
\end{proof}	

It follows that $\cong_u$ is a complete analytic set and hence non-Borel. 
 
 The rest of this section will establish the following result of independent interest, which is the first step towards proving our Main Result~\ref{th:main}.
 \begin{theorem} \label{th:1} $\ell_\infty$ is Borel reducible to $  \cong_u$, namely, uniform homeomorphism of  the ultrametric  spaces denoted by points in $\mathcal T$. Consequently, $\cong_u$ is not Borel reducible to any orbit equivalence relation induced by a Borel action of a Polish group. In particular, $\cong_u$ is not Borel reducible to $\mbox{\rm GI}$.
 \end{theorem}
 
\subsection{Some facts on uniform homeomorphisms}  
 The following is a direct consequence of the definition \ref{def:tree} of the metric on path spaces, and the definition of  uniform continuity.
 
 \begin{lemma}\label{lem:basic} Let $T, S$ be pruned subtrees of $\mathbb{N}^{<\mathbb{N}}$ and let $\Phi\colon [T]\to [S]$. Then $\Phi$ is uniformly continuous if and only if 
 \[
 \forall n\in\mathbb{N}\;\exists m\in\mathbb{N}\;\forall x,y\in[T]\;
 \bigl(x\!\upharpoonright\! m = y\!\upharpoonright\! m
 \;\rightarrow\;
 \Phi(x)\!\upharpoonright\! n = \Phi(y)\!\upharpoonright\! n\bigr).
 \]

 \end{lemma}
 
 In this context we define $\Omega_\Phi(n)$ to be the least $m$ such that the condition in the above lemma holds. 
 Note that the lemma implies
 $$ \big|\big\{ t\in T\colon |t|=\Omega_\Phi(n)\big\}\big|\geq \big|\big\{s\in S\colon |s|=n\big\}\big|. $$

 \begin{lemma}\label{lem:unbdd} Let $T, S$ be pruned subtrees of $\mathbb{N}^{<\mathbb{N}}$ and let $\Phi\colon [T]\to [S]$ be  a uniform homeomorphism (see \ref{def:UH}).   If both  $[T], [S]$ are uncountable, then $\Omega_\Phi$ and $\Omega_{\Phi^{-1}}$ are unbounded monotone functions from $\mathbb{N}$ to $\mathbb{N}$.
 \end{lemma}
 
 \begin{proof} The monotonicity is obvious. If $\Omega_\Phi$ is bounded by $m$, then for each $n $ the string $\Phi(x)\!\!\upharpoonright\!n$ is determined by $x \!\!\upharpoonright\! m$. So   $[S]$ is countable. 
 \end{proof}

\subsection{Definition of the reduction}

We will need  to restrict the entries of the sequences in $\mathbb R^\mathbb{N}$ to natural numbers. This does not decrease the complexity of the equivalence relation:
\begin{lemma}\label{lemma_th:1} $\ell_\infty\leq_B \ell_\infty\!\!\upharpoonright\!\mathbb{N}^\mathbb{N}$.
\end{lemma}

\begin{proof} Define $f\colon \mathbb{R}^\mathbb{N}\to \mathbb{N}^\mathbb{N}$ by
	$$ f(x)(2n)=\left\{\begin{array}{ll}\lfloor x(n)\rfloor, & \mbox{if $x(n)\geq 0$,} \\ 0, & \mbox{otherwise,} \end{array}\right. $$
	and
	$$ f(x)(2n+1)=\left\{\begin{array}{ll}0, & \mbox{if $x(n)\geq 0$,} \\ -\lfloor x(n)\rfloor, & \mbox{otherwise.} \end{array}\right. $$
	Then for any $x, y\in\mathbb{R}^\mathbb{N}$ and $n\in\mathbb{N}$, we have $|f(x)(n)-f(y)(n)|\leq |x(n)-y(n)|$. Thus, if $x\ell_\infty y$, then $f(x)\ell_\infty f(y)$. Conversely, if $|f(x)(n)-f(y)(n)|< M$ for all $n\in\mathbb{N}$, then for all $n\in\mathbb{N}$, either $x(n)y(n)>0$ and $|x(n)-y(n)|<M$ hold, or else $x(n)y(n)\leq 0$ and $|x(n)|, |y(n)|<M$ hold; in both cases we have $|x(n)-y(n)|<2M$. Thus $f$ is a reduction from $\ell_\infty$ to $\ell_\infty\!\!\upharpoonright\!\mathbb{N}^\mathbb{N}$.
\end{proof}

For the next step of the reduction, we consider an arbitrary $x\in \mathbb{N}^\mathbb{N}$ and define a pruned tree $\mathsf{T}_x$ on $\mathbb{N}$. Before defining $\mathsf{T}_x$, we first introduce  some auxiliary trees 
$$ \mathsf{S}_*, \mathsf{S}_1, \mathsf{S}_2, \dots, \mathsf{S}_k, \dots, \mathsf{S}_\omega, \mathsf{R}_0, \mathsf{R}_1, \dots, \mathsf{R}_k, \dots $$
that will be used as building blocks.
Let $\mathsf{S}_*\subseteq \mathbb{N}^{<\mathbb{N}}$ be the tree consisting of all finite sequences $s\in \mathbb{N}^{<\mathbb{N}}$ such that $s(n)$ is even  for each $n\in\mathbb{N}$. Then $\mathsf{S}_*$ is a copy of the full tree $\mathbb{N}^{<\mathbb{N}}$. The point of defining $\mathsf{S}_*$ is that we may arbitrarily extend $\mathsf{S}_*$ by adding extensions of existing nodes, and still maintain that the extended tree be an ``official"  subtree of $\mathbb{N}^{<\mathbb{N}}$. 

The trees $\mathsf{S}_1, \mathsf{S}_2, \dots, \mathsf{S}_k,\dots$ will be abstractly and inductively defined. Their exact realizations as subtrees of $\mathbb{N}^{<\mathbb{N}}$ will not be important. To begin, let $\mathsf{S}_1$ be a single branch. Thus the Cantor--Bendixson rank of $[\mathsf{S}_1]$ is $1$. Let $\mathsf{S}_2$ be a pruned tree so that the Cantor--Bendixson derivative of $[\mathsf{S}_2]$ is $[\mathsf{S}_1]$. So the Cantor--Bendixson rank of $[\mathsf{S}_2]$ is $2$. In general, let $\mathsf{S}_{k+1}$ be a pruned tree whose Cantor--Bendixson derivative is $[\mathsf{S}_k]$. Then the Cantor--Bendixson rank of $\mathsf{S}_{k+1}$ is $k+1$. This finishes the definitions of $\mathsf{S}_k$ for $1\leq k<\mathbb{N}$.

Next we define $\mathsf{S}_\omega$ to be the extension of $\mathsf{S}_*$ defined as follows. For each $k\in\mathbb{N}$, let $\mathsf{z}_k$ be the branch of $\mathsf{S}_*$ defined as
%
$$ \mathsf{z}_k(m)=\left\{\begin{array}{ll}2k, & \mbox{if $m=0$,} \\0, & \mbox{otherwise.} \end{array}\right. $$
For any $k, m\in\mathbb{N}$, let $\mathsf{t}_{k,m}=\mathsf{z}_k\!\!\upharpoonright\!(m+1)$. $\mathsf{S}_\omega$ is obtained by extending $\mathsf{S}_*$ with appending a copy of $\mathsf{S}_k$ at the node $\mathsf{t}_{k,2m}$ for all $k, m\in\mathbb{N}$. 

The distinctive feature of $\mathsf{S}_\omega$ is that any autohomeomorphism of $[\mathsf{S}_\omega]$ must keep~$\mathsf{z}_k$ fixed for all $k\in\mathbb{N}$. This is because of the topological characterization of $\mathsf{z}_k$ as the only   $z \in  [\mathsf{S}_\omega]$ with the   property that:
%
\begin{quote} for any clopen set $V$ with $z\in V$, there exists a clopen set $U\subseteq V$ with $z\in U$ such that the $k$-th Cantor--Bendixson derivative of $U$ is the perfect kernel of $U$, and for any $i<k$, the $i$-th Cantor--Bendixson derivative of $U$ contains a perfect subset but is not itself perfect.
\end{quote}

Next we abstractly define a sequence of finitely splitting pruned trees 
$$\mathsf{R}_0, \mathsf{R}_1,\dots, \mathsf{R}_k,\dots$$
so that no   $[\mathsf{R}_k]$ has an isolated point,  as follows.
 For each $k\in\mathbb{N}$, let $\mathsf{R}_k$ be $2^k$-splitting on the first level, and then each node on the first level be extended with a copy of the full binary tree. It is easy to see that for any $k\in\mathbb{N}$, $[\mathsf{R}_k]$ is homeomorphic to a Cantor set, and therefore there is a uniform homeomorphism between $[\mathsf{R}_k]$ and $[\mathsf{R}_\ell]$ for any $k, \ell\in\mathbb{N}$. 

Finally we are ready to define $\mathsf{T}_x$ given an $x\in\mathbb{N}^\mathbb{N}$. $\mathsf{T}_x$ is obtained from $\mathsf{S}_\omega$ by extending the node $\mathsf{t}_{k, 2m+1}$ by a copy of $\mathsf{R}_{x(k)}$ for all $k, m\in\mathbb{N}$. This finishes the definition of $\mathsf{T}_x$. Figure~\ref{fig:Tx} illustrates the definition of $\mathsf{T}_x$.

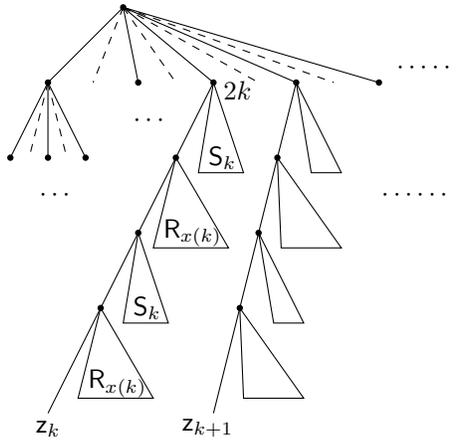
\begin{figure}[h]
\centering

  \begin{tikzpicture}[every node/.style={font=\small}]
  	
  	\coordinate (R) at (0,0);
  	\fill (R) circle (1.2pt);
  	
  	\draw (R) -- (-1,-1);
	\draw (R) -- (0.2,-1);
  	\draw (R) -- (1.2,-1);
  	\draw (R) -- (2.3,-1);
	\draw (R) -- (3.4, -1);
	\node[right] at (3.5,-0.8) {$\cdots\cdots$};
  	
  	\draw[dashed] (R) -- (-0.4,-1);
  	\draw[dashed] (R) -- (0.7, -1);
  	\draw[dashed] (R) -- (1.8,-1);
	\draw[dashed] (R) -- (2.8, -1);

	\coordinate (A) at (-1,-1);
	\coordinate (B) at (1.2, -1);
	\coordinate (C) at (2.3,-1);
	\coordinate (D) at (3.4,-1);
	\coordinate (E) at (0.2,-1);
	\fill (A) circle (1.2pt);
	\fill (B) circle (1.2pt);
	\fill (C) circle (1.2pt);
	\fill (D) circle (1.2pt);
	\fill (E) circle (1.2pt);

	\coordinate (F) at (-1.5,-2);
	\coordinate (G) at (-1, -2);
	\coordinate (H) at (-0.5, -2);
	\fill (F) circle (1.2pt);
	\fill (G) circle (1.2pt);
	\fill (H) circle (1.2pt);

	\draw (A) -- (F);
	\draw (A) -- (G);
	\draw (A) -- (H);
	\draw[dashed] (A) -- (-1.25,-2);
	\draw[dashed] (A) -- (-0.75, -2);
	\node[right] at (-1.25,-2.5) {$\cdots$};
	\node[right] at (0, -1.5) {$\cdots$};
  	
  	\draw (B) -- (-1,-5.4);
	\node[right] at (1.2, -1.1) {\small $2k$};
	\node[left]  at (-0.7,-5.6) {$\mathsf{z}_k$};
	\node[left]  at (1.6,-5.6) {$\mathsf{z}_{k+1}$};

  	\draw (B) -- (1,-2.2) -- (1.6, -2.2) --(B);
  	\node[right] at (1,-2.0) {$\mathsf{S}_k$};

	\draw (0.2,-3) -- (0,-4.2) -- (0.6, -4.2) -- (0.2,-3);
	\node[right] at (0,-4) {$\mathsf{S}_k$};
  	
 	\fill (0.7,-2) circle (1.2pt); 	
	\fill (0.2, -3) circle (1.2pt);
	\fill (-0.3, -4) circle (1.2pt);

	\draw (0.7,-2) -- (0.4,-3.2) -- (1.4,-3.2) -- (0.7,-2);
	\node[right] at (0.4, -3) {$\mathsf{R}_{x(k)}$};

  	\draw (-0.3,-4) -- (-0.6,-5.2) -- (0.4,-5.2) -- (-0.3,-4);
	\node[right] at (-0.6, -5) {$\mathsf{R}_{x(k)}$};
  	
  	\draw (C) -- (1.2, -5.4);
  	
  	\draw (C) -- (2.5,-2.2) -- (2.9, -2.2) --(C);
  	\draw (1.8, -3) -- (2,-4.2) -- (2.4, -4.2) --(1.8,-3);
	\fill (1.8,-3) circle (1.2pt);

 	\draw (2.05,-2) -- (2.1,-3.2) -- (2.9, -3.2) --(2.05,-2);
  	\draw (1.55, -4) -- (1.6,-5.2) -- (2.4, -5.2) --(1.55,-4);
	\fill (2.05,-2) circle (1.2pt);	
	\fill (1.55,-4) circle (1.2pt);

	\node[right] at (3.3,-2.5) {$\cdots\cdots$};

  \end{tikzpicture}
\caption{\label{fig:Tx} The construction of $\mathsf{T}_x$.}
\end{figure}

For each $x\in \mathbb{N}^\mathbb{N}$, we let $\mathsf{S}_*(\mathsf{T}_x)$, $\mathsf{S}_\omega(\mathsf{T}_x)$, $\mathsf{t}_{k,m}(\mathsf{T}_x)$, $\mathsf{z}_k(\mathsf{T}_x)$ etc. denote respectively the subtrees $\mathsf{S}_*$, $\mathsf{S}_\omega$, nodes $\mathsf{t}_{k,m}$, and branches $\mathsf{z}_k$ etc.  in the construction of $\mathsf{T}_x$. Note that for different $x$ these objects are isomorphic, but we need to address the specific copies of them as substructures of $\mathsf{T}_x$.

 For $k,m\in\mathbb{N}$, we let $\mathsf{R}_{x(k), m}$ denote the copy of $\mathsf{R}_{x(k)}$ which was appended at the node $\mathsf{t}_{k,2m+1}(\mathsf{T}_x)$. 

Note that for any $x\in \mathbb{N}^\mathbb{N}$ and $k\in\mathbb{N}$, $\mathsf{z}_k(\mathsf{T}_x)$ is still topologically characterized by the same property as above, since the additional sets added beyond $[\mathsf{S}_\omega]$ in the construction of $[\mathsf{T}_x]$ are all perfect.

 Therefore, if $\theta$ is any homeomorphism from $[\mathsf{T}_x]$ to $[\mathsf{T}_y]$, then we must have that $\theta(\mathsf{z}_k(\mathsf{T}_x))=\mathsf{z}_k(\mathsf{T}_y)$ for all $k\in\mathbb{N}$. 

Although we do not need this in the proof, we observe that for any $x, y\in \mathbb{N}^\mathbb{N}$, $[\mathsf{T}_x]$ and $[\mathsf{T}_y]$ are homeomorphic.

 \begin{definition} Let $(X, d_X)$ and $(Y, d_Y)$ be Polish metric spaces and let $f\colon X\to Y$. Let $L>0$. We say that $f$ is {\em $L$-Lipschitz} if for any $x, x'\in X$, 
	$$d_Y(f(x), f(x'))\leq L\, d_X(x,x'). $$
	$f$ is {\em Lipschitz} if it is $L$-Lipschitz for some $L>0$.
\end{definition}

For $k, \ell\in\mathbb{N}$, consider the natural homeomorphism $\Phi_{k,\ell}$ between $[\mathsf{R}_k]$ and $[\mathsf{R}_{\ell}]$. For later   use we note that $\Phi_{k,\ell}$ is $L$-Lipschitz with $L=2^{\ell-k}$. Similarly, $\Phi_{k,\ell}^{-1}=\Phi_{\ell,k}$ is $L^{-1}$-Lipschitz. 

%

\subsection{Why the reduction works}

Assume $x, y\in \mathbb{N}^\mathbb{N}$ with $x\ell_\infty y$. Let $M>0$ be such that $|x(n)-y(n)|<M$ for all $n\in\mathbb{N}$. We define a uniform homeomorphism $\Phi$ between $[\mathsf{T}_x]$ and $[\mathsf{T}_y]$. For this, note that $[\mathsf{T}_x]$ is the disjoint union of $[\mathsf{S}_{\omega}(\mathsf{T}_x)]$ and $[\mathsf{R}_{x(k),m}]$ for all $m\in\mathbb{N}$, and $[\mathsf{T}_y]$ is the disjoint union of $[\mathsf{S}_{\omega}(\mathsf{T}_y)]$ and $[\mathsf{R}_{y(k),m}]$ for all $m\in\mathbb{N}$. Let $\Phi$ be the homeomorphism from $[\mathsf{S}_{\omega}(\mathsf{T}_x)]$ onto $[\mathsf{S}_{\omega}(\mathsf{T}_y)]$ induced by the identity map from $\mathsf{S}_{\omega}(\mathsf{T}_x)$ to $\mathsf{S}_{\omega}(\mathsf{T}_y)$. Then it is obvious that $\Phi$ is an isometry, and therefore both $\Phi$ and $\Phi^{-1}$ are $1$-Lipschitz. Let $\psi_{k, m}$ be the natural homeomorphism between $[\mathsf{R}_{x(k),m}]$ and $[\mathsf{R}_{y(k),m}]$. Then for all $k,m\in\mathbb{N}$, both $\psi_{k,m}$ and $\psi_{k,m}^{-1}$ are $2^M$-Lipschitz. Let $\Psi=\Phi\cup\bigcup_{k,m}\psi_{k,m}$. Then both $\Psi$ and $\Psi^{-1}$ are $2^M$-Lipschitz. Thus $\Psi$ is a uniform homeomorphism between $[\mathsf{T}_x]$ and $[\mathsf{T}_y]$.

Conversely, assume $\Psi$ is a uniform homeomorphism between $[\mathsf{T}_x]$ and $[\mathsf{T}_y]$. In particular, $\Psi$ is a homeomorphism. Then  $\Psi(\mathsf{z}_k(\mathsf{T}_x))=\mathsf{z}_k(\mathsf{T}_y)$ for any $k\in\mathbb{N}$. By Lemma~\ref{lem:unbdd}, there is $m>1$ such that 
$$ \Omega_{\Psi}(2m+2)>1 $$
and, letting $n=\Omega_{\Psi}(2m+2)$, also
$$ \Omega_{\Psi^{-1}}(2n+2)>1. $$
Let $M$ be the least such that $2M+2\geq\Omega_{\Psi^{-1}}(2n+2)$. 
We claim that for any $k\in\mathbb{N}$, $\Psi([\mathsf{R}_{x(k),n}])$ is contained in 
$$\bigcup_{m<p\leq M}[\mathsf{R}_{y(k),p}]. $$
To see this, let $u\in [\mathsf{R}_{x(k),n}]$ be arbitrary. Then 
$$u\upharpoonright (2n+1)=\mathsf{z}_k(\mathsf{T}_x)\upharpoonright (2n+1). $$
Since $2n+1>n=\Omega_{\Psi}(2m+2)$, we have that 
$$ \Psi(u)\upharpoonright (2m+2)=\mathsf{z}_k(\mathsf{T}_y)\upharpoonright (2m+2). $$
However, if $\Psi(u)=v\in [\mathsf{R}_{y(k),p}]$ for $p>M$, then 
$$ v\upharpoonright (2p+1)=\mathsf{z}_k(\mathsf{T}_y)\upharpoonright (2p+1), $$
and since $2p+1>2M+2\geq \Omega_{\Psi^{-1}}(2n+2)$, we have
$$ u\upharpoonright (2n+2)=\mathsf{z}_k(\mathsf{T}_x)\upharpoonright (2n+2), $$
which is a contradiction.

By our construction, the number of nodes at level $2n+2$ in the copy of $\mathsf{R}_{x(k),n}$ is $2^{x(k)}$. The number of nodes at level $2M+2$ in the copies of $\bigcup_{m<p\leq M} \mathsf{R}_{y(k),p}$ is no more than
$$ 2^{y(k)+(M-m)}+2^{y(k)+(M-m-1)}+\cdots+2^{y(k)} = 2^{y(k)}(2^{M-m+1}-1)\leq 2^{y(k)+M-m+1}.
$$
By Lemma~\ref{lem:basic} we have that $2^{y(k)+M-m+1}\geq 2^{x(k)}$, and hence $x(k)-y(k)\leq M-m+1$. This means that $x(k)-y(k)$ is bounded. By symmetry, $y(k)-x(k)$ is also bounded. Therefore $x\ell_\infty y$.

\subsection{Further remarks}\label{Sec_fr}


\begin{definition} We call an analytic equivalence relation {\em universal} if any analytic equivalence relation is Borel reducible to it. 
\end{definition}

Such equivalence relations were called complete analytic equivalence relations in~\cite{FLR}. Here we call them universal to avoid possible confusion. In \cite{FLR} and subsequent research, many natural equivalence relations were found to be universal analytic equivalence relations. In particular, it was shown in \cite{FLR} that the uniform homeomorphism for all Polish metric spaces is a universal analytic equivalence relation. 

  We remark that our result does not show that $\cong_u$ is universal. By \ref{prop:quicko} GI is reducible to $\cong_u$, and by the main result~\ref{th:main} $\ell_\infty$ is reducible to $\cong_u$. However,   the equivalence relation $c_0$ does not reduce to $\mbox{\rm GI} \times \ell_\infty$. Recall that $c_0$ is the equivalence relation on $\mathbb{R}^{\mathbb{N}}$ defined as
$$ (x_n)c_0(y_n)\iff \lim_n |x_n-y_n|=0. $$
Toward a contradiction, assume that there is a map $f$ witnessing that $c_0$ is Borel reducible to $\mbox{\rm GI}\times \ell_\infty$. The map $f$ induces a homomorphism of equivalence relations from $c_0$ to $\mbox{\rm GI}$. By the Hjorth turbulence theorem (see \cite[Theorem 3.18 and Example 3.23]{Hj}; also see \cite[Theorems 10.4.2 and 10.5.2]{GaoBook}), there is a $c_0$-invariant comeager set $C$ on which $c_0$ is sent by $f$ to a single $\mbox{\rm GI}$-class. This implies that $c_0$ on $C$ is Borel reducible to $\ell_\infty$. However, $c_0\upharpoonright C^2$ is ${\bf\Pi}^0_3$-complete (see \cite[Exercise 8.5.1]{GaoBook}) and as such, it is not Borel reducible to any $F_\sigma$ equivalence relation (see \cite[Theorem 8.5.2]{GaoBook} and its proof). 

We conjecture that $\cong_u$, and hence uniform homeomorphism of  Polish ultrametric spaces, are universal analytic equivalence relations.

\section{Preliminaries on procountable groups}
In \ref{def:procountable} we defined procountable groups.  This section provides   material on procountable groups that will be needed for the second step in our overall proof of the Main Result \ref{th:main}.
 \subsection{Presentations of procountable groups} \label{s:Procountable}
 Encoding inverse systems by   structures with domain $\mathbb N$ that have   infinitely many sorts, we obtain a  standard Borel space  $\+ P$ of presentations of  procountable   groups. We   show     that these presentations can be converted in a Borel way into  procountable  closed subgroups of $\S$, and  that the latter form a Borel set~$\+ C$.

\begin{notation} \label{not:strings} {\rm  The space of all finite sequences of natural numbers is denoted as $\mathbb{N}^{<\mathbb{N}}$. For $s, t\in \mathbb{N}^{<\mathbb{N}}$ and $x\in\mathbb{N}^{\mathbb{N}}$, we use $s\subseteq t$ to denote that $s$ is an initial segment of $t$ and similarly for $s\subseteq x$. For $s\in\mathbb{N}^{<\mathbb{N}}$, we write \[[s]= \left\{ x \in \S \colon \, s\subseteq x\right\}.\] For  a closed subgroup $G \leq_c \S$, we  will write  $[s]_G$ for $[s] \cap G$.  We let $\sigma_n$  denote the tuple $(0, 1 , \ldots, n-1)$; note that the $[\sigma_n]_G$ form a neighbourhood basis of the identity in $G$ consisting of open subgroups.   
		
\smallskip \noindent 	
		Given $s, t\in\mathbb{N}^{<\mathbb{N}}$, let $s\circ t$ be the sequence representing  composition, on the longest initial segment of $\mathbb{N}$ in which it is defined. Let $s^{-1}$ be the inverse of $s$ on the longest initial segment in which it is defined.  For instance, $(7,4,3,1,0 ) \circ (3,4 ,6 ) = (1,0)$, and $(1,2,0,5)^{-1} = (2,0,1)$.} \end{notation}

\smallskip

\noindent \textit{From presentations in $\+ P$ to procountable closed subgroups of $\S$.}
As defined in  \ref{def:procountable}, 
an inverse limit $G=  \varprojlim_n (G_n,p_n) $ can be concretely thought of as a closed subgroup of $\prod_n G_n$ where each $G_n$ is discrete, i.e., \[G = \left\{ x \in \prod_n G_n \colon \, \forall n \, [p_n(x(n+1))= x(n)]\right\}.\] 

\begin{remark} \label{rem:diagram to closed subgroup}
	From an inverse system as above,  one  can in a Borel way determine a corresponding closed subgroup of $\S$. One  replaces $G_n$ by an isomorphic group with domain $D_n=\{ \langle  i,n \rangle \colon i \in \mathbb N \}$,  where $\langle\cdot,\cdot\rangle\colon \mathbb{N}\times\mathbb{N}\to \mathbb{N}$ is the Cantor pairing function.  Then   the left regular presentation  of $G_n$ is a copy of $G_n$ as a permutation group with support  $D_n$, and the    product of these copies in $\S$  is a copy $C$ of $\prod_n G_n$ as a closed subgroup of $\S$. Finally, one obtains a copy of $G$ as a closed subgroup of  $C$ in a Borel way from the binding maps $p_n$. 
\end{remark}

%
%
%
%
%
%
%

 \begin{proposition}  \label{prop: Borel pc} The  procountable closed subgroups of $\S$   form a  Borel class~$\+ C$. \end{proposition}
 
 \begin{proof}  
As mentioned before, $G$ is procountable if and only if the identity has a neighbourhood basis of normal open  subgroups. In the setting of closed subgroups of $\S$ this is equivalent to: for each $n$, there is $k>n$ such that the normal closure (in $G$) of  $[\sigma_k]_G$ is contained in $[\sigma_n]_G$. Since  $[\sigma_n]_G$ is a subgroup, it suffices to express as a Borel condition about $G$ that
 	\begin{equation} \label{eqn:Gcirc} \forall x \in G \, \left[x^{-1} \circ ( [\sigma_k]_G) \circ x\subseteq [\sigma_n]_G \right].  \end{equation} 
 	In the following, letters $s, t$ etc. range over sequences that have an extension in $G$.  Recalling \ref{not:strings}, we claim that  (\ref{eqn:Gcirc}) is equivalent to the Borel condition
 	\begin{equation} \label{eqn:Gcirc2} \forall  s\ \forall t \supseteq \sigma_k \ \  [s^{-1} \circ t  \circ s ]_G \cap [\sigma_n]_G \neq \emptyset.  \end{equation} 
 
\smallskip \noindent 	
 	For (\ref{eqn:Gcirc}) $\to$ (\ref{eqn:Gcirc2}), given $s$ and $t\supseteq \sigma_k$, let $x \supseteq s$, $x\in G$ and $y \supseteq t$, $y \in G$. Then $x^{-1} \circ y \circ x \in [\sigma_n]_G$. 
 	
\smallskip \noindent 	
	For (\ref{eqn:Gcirc2}) $\to$ (\ref{eqn:Gcirc}), suppose that there is $x \in G$ such that $x^{-1} \circ ( [\sigma_k]_G) \circ x$ is not contained in $[\sigma_n]_G$. Then, by the continuity of the conjugation operation, for sufficiently long $s\subseteq x$ and sufficiently long $t \supseteq \sigma_k$, we have $[s^{-1} \circ t  \circ s   ]_G \cap [\sigma_n]_G= \emptyset$. 
 \end{proof}
 
 \begin{remark}
 	As a converse to Remark~\ref{rem:diagram to closed subgroup}, from $G \in \+ C$,  one can in a Borel way obtain  an inverse system $(G_n, p_n)_{n \in \mathbb N} $ such that $G\cong \varprojlim_n (G_n, p_n)$.   This uses the methods in the proof above and standard tools from descriptive set theory. We omit the details as this fact will not be needed below. 
 \end{remark}

\subsection{Pro$_\omega$ categories}We provide some  preliminaries on   the correspondence of isomorphisms of procountable groups and certain isomorphisms of inverse systems of type $\omega$. Our purpose is mainly   to lay down  a notational framework; more general results are obtained in~\cite{Mardesic.Segal} using a somewhat different language.   

\begin{definition}\label{def_pro}  Let  $\mathcal{K}$ be a category that has  as objects countable structures in the same signature,    and as morphisms surjective homomorphisms  between them. 
	The~category $\mrm{Pro}_\omega( \mathcal{K})$ has as objects the  inverse systems $( \ol A , \ol p)   =   (A_n, p_{n})_{n \in \mathbb N}   $ over~$\+ K$, where $p_n \colon A_{n+1} \to A_n$. The pre-morphisms 
	\[\langle \ol f, \phi \rangle \colon( \ol A , \ol p)  \to (\ol B, \ol q )\] are given as sequences of       morphisms $f_n: A_{\phi(n) } \to B_n$ where $\phi: \mathbb{N} \rightarrow \mathbb{N}$ is increasing, and all the diagrams   commute. In more detail, for $k> n$, write $p_{k,n}$ for the ``interval composition"  maps $p_n \circ p_{n+1} \circ \ldots \circ p_{k-1} \colon A_k \to A_n$ and define $q_{k,n}$ in a similar way;  we require that the 
	diagrams
	\begin{center}
		\begin{tikzcd}
			A_{\phi(k)} \arrow[r, "f_k"] \arrow[d, "p_{\phi(k), \phi(n)}"'] & B_k \arrow[d, "q_{k,n}"] \\
			A_{\phi(n)} \arrow[r, "f_n"] & B_n
		\end{tikzcd}
	\end{center}   commute. Composition of pre-morphisms is defined in the canonical way (see \cite[pg. 5]{Mardesic.Segal}), and the identity pre-morphism is $({\mrm{id}_{\ol{A}}}, \mrm{id}_{\mathbb N})$, where $\mrm{id}_{\ol{A}} =
	 \{{\mrm{id}_{A_n}}\}_{n \in \mathbb{N}}$. 
	
	\smallskip \noindent	
	For each pair of objects   $( \ol A , \ol p)  $ and $  ( \ol B , \ol q)$ one defines an equivalence relation $\sim$ between pre-morphisms $(\bar{f}, \phi) ,(\bar{f}', \phi'): ( \ol A , \ol p)   \to ( \ol B , \ol q)$: let  $(\bar{f}, \phi) \sim (\bar{f}', \phi')$ if for  each $n \in \mathbb{N}$ there is $m \in \mathbb{N}$ such that $m \geq \phi(n), \phi'(n)$ and the following diagram commutes:
	\begin{equation*}
		\begin{tikzcd}
			&  A_m \arrow[ld, "p_{m,\phi(n)}"']\arrow[rd, "p_{m,\phi'(n)}"]   & \\
			A_{\phi(n)} \arrow[rd, "f_n"' below] & & A_{\phi'(n)} \arrow[ld, "f'_n"] \\
			& B_n &
		\end{tikzcd}
	\end{equation*}
	The transitivity of  the equivalence relation $\sim$ and its  compatibility 
	with composition of pre-morphisms  are proved in \cite[pg.\ 6]{Mardesic.Segal}. The morphisms of $\mrm{Pro}_\omega( \mathcal{K})$ are   defined as  the equivalence classes of pre-morphisms under $\sim$. 
	
\end{definition}

\begin{remark}\label{useful_remark} Notice that for a pre-morphism $(\bar{f}, \phi): ( \ol A , \ol p) \to ( \ol A , \ol p)$, we have $(\bar{f}, \phi) \sim ({\mrm{id}_{\ol{A}}}, \mrm{id}_{\mathbb N})$ if and only if for every $n \in \mathbb{N}$ there is $m \geq \phi(n), n$ such that the   diagram  above commutes where  $\phi'(n)= n$, $B_n= A_n$ and $f_n'$ is the identity of $A_n$.
 Since the morphisms of $\+ K$ are surjections, this simply means 	that  $f_n = p_{\phi(n), n}$ for each $n$.
\end{remark}

Let $\Gps$ be the category of countable groups with onto homomorphisms. 
\begin{lemma}\label{preserving_iso_progroups} \mbox{ }
	
	\begin{enumerate}[(i)]
		\item The inverse limit operation yields  a full functor $P$ from  $\mrm{Pro}_\omega(\Gps)$ to the category of procountable groups with continuous homomorphisms that have dense range.  
		\item $P$ preserves isomorphisms in both directions.  
	\end{enumerate}
\end{lemma}

\begin{proof} All inverse systems will be of type $\omega$.
	
	\noindent  (i).  For  an   inverse system $(\ol C, \ol r) $ over $\Gps$,  let $P(\ol C, \ol r)$ denote the inverse limit $\varprojlim \, (\ol C, \ol r) $. For  inverse systems $(\ol A , \ol p)$,  $(\ol B, \ol q)$ and   a pre-morphism $(\bar f , \phi ) \colon \ol A \to \ol B$, write $G = P  (\ol A , \ol p)$ and $H = P (\ol B, \ol q)$, and  let $P(\bar f, \phi)\colon G \to H  $ be obtained by the maps $G \to A_{\phi(n)} \to B_n$  using the universal property of $H$. It is  easily verified that $P(\bar f, \phi)$  has dense range,  and that  the assignment $(\bar f , \phi ) \to P(\bar f , \phi )$ is functorial. So we have a functor that clearly induces a functor on the quotient category $\mrm{Pro}_\omega(\Gps)$.

	\smallskip 
	\noindent 
	In order to show that the functor $P$ is full, suppose that $\gamma \colon G\to H$ is a continuous homomorphism with dense range. We will define  a pre-morphism $(\bar f , \phi ) \colon (\ol A, \ol p)  \to (\ol B, \ol q) $  with $P(\bar f, \phi ) = \gamma$. 
	Let $a_n   $ and $b_n  $ be the kernels of the canonical projections $\pi^G_n: G \to A_n$ and $\pi^H_n : H \to B_n$ respectively.  Let $\phi$ be an increasing function such that $a_{\phi(n)} \subseteq \gamma^{-1}(b_n)$. Such $\phi$ exists because  $\gamma^{-1}( {b_n})$ is an open subgroup of $A_n$, and the $  {a_n}$ form a basis of neighbourhoods of $e$. (This part of the argument is special to procountable groups.) For $n \in \mathbb N$ we define $f_n: A_{\phi(n)} \rightarrow B_n$ to be the unique map such that  the   diagram
	\begin{center}
		\begin{tikzcd}
			G \arrow[r, "\gamma"] \arrow[d, "\pi^G_{\phi(n)}", left]& H \arrow[d, "\pi^H_n"] \\
			A_{\phi(n)} \arrow[r, "f_n"] & B_n \ .
		\end{tikzcd}
	\end{center}
commutes. Since $\gamma$ has dense range, $\pi^H_n \circ \gamma$ is onto, and hence $f_n$ is onto. 
To verify that for each $k >n$      the   diagram 
	\begin{center}
		\begin{tikzcd}
			A_{\phi(k)} \arrow[r, "f_k"] \arrow[d, "p_{\phi(k), \phi(n)}"'] & B_k \arrow[d, "q_{k,n}"] \\
			A_{\phi(n)} \arrow[r, "f_n"] & B_n  \  .
		\end{tikzcd}
	\end{center}
	commutes, it suffices to consider    the   diagram	
	\begin{center}
		\begin{tikzcd}
			G \arrow[r, "\gamma"]  \arrow[d, "\pi^G_{\phi(k)}", left]  & H \arrow[d, "\pi^H_k"] \\
			A_{\phi(k)} \arrow[r, "f_k"]\arrow[d, "p_{\phi(k), \phi(n)}"', left] & B_k \arrow[d, "q_{k,n}", right] \\
			A_{\phi(n)} \arrow[r, "f_n"] & B_n
		\end{tikzcd}
	\end{center}
	and realise  that the upper rectangle and the outer rectangle both   commute, noting that $p_{\phi(k), \phi(n)} \circ \pi^G_{\phi(k)} = \pi^G_{\phi(n)}$ and $q_{k,n} \circ \pi^H_{k} = \pi^H_n$.
	
	
	\medskip 
\noindent	(ii).  For the   forward direction, consider an isomorphism
	$$(\bar f , \phi ) \colon \ol A \rightleftarrows \ol B: (\bar g , \psi )$$
	in $\mrm{Pro}_\omega(\Gps)$. Then the composition $(\bar g , \psi ) \circ (\bar f , \phi )$, which we recall   equals  
	$$A_{\phi(\psi(n))} \xrightarrow{f_{\psi(n)}} B_{\psi(n)} \xrightarrow{g_n} A_n,$$
	is such that $(\bar g , \psi ) \circ (\bar f , \phi ) \sim (\mrm{id}_{\ol{A}}, \mrm{id}_{\mathbb N})$. By~\ref{useful_remark}, this simply means that $g_n \circ f_{\psi(n)} = p_{\phi(\psi(n)), n}$. We  argue similarly for the composition $(\bar f , \phi ) \circ (\bar g , \psi )$. Using this, it is easy to see that the following is an isomorphism of topological groups
	$$P(\bar f , \phi ) \colon P(\ol A) \rightleftarrows P(\ol B): P(\bar g , \psi ).$$
	For the converse  direction, consider an isomorphism of procountable groups:
	$$\gamma: G \rightleftarrows H : \delta.$$
	Using the fact that the functor $P$ is full,  we can rewrite this as follows (writing $\ol A$ in place of $(\ol A, \ol p) $ etc.) 
	$$P(\bar f , \phi ) \colon P(\ol A) \rightleftarrows P(\ol B): P(\bar g , \psi ),$$
	where   $G = \colon P(\ol A), H = P(\ol B)$ and $P(\bar f , \phi ) = \gamma, P(\bar g , \psi ) = \delta$. It is   easy to verify  that 
	$$(\bar f , \phi ) \colon \ol A \rightleftarrows \ol B: (\bar g , \psi )$$
	is an isomorphism in $\mrm{Pro}_\omega(\Gps)$.
\end{proof}

\section{Borel reduction of uniform homeomorphism to isomorphism of   procountable groups} \label{s:2}

	In this section we implement the second step of our proof strategy toward a proof of our Main Result \ref{th:main}, as laid out in the introduction. Definition~\ref{def:GT}   will provide a Borel reduction from uniform homeomorphism on Polish ultrametric spaces to isomorphism of   procountable groups.

 \subsection{The  groups $L(V)$} \label{ss:LV}
For a countably infinite  set $V$, we define the free product of cyclic groups of size $2$ with generators in $V$ by 
\[L(V) = \langle  s \mid s^2 = e \ \ ( s \in V) \rangle.\] (This group is also known as the free Coxeter group of rank~$\aleph_0$.) By the general theory of free products of group (see, e.g., \cite[IV.4.7]{LyndonSchupp1977}), each involution is conjugate to a generator, and no two generators are conjugate.  If $r \colon V \to W$, then we have a  map $\widehat r \colon L(V) \to L(W)$ induced by $v \mapsto r(v)$.

\subsection{The  reduction for the second step}
\begin{definition} \label{def:GT}
Given a pruned tree $T$, let $T_n = \{ s \in T \colon |s|= n\}$ be the $n$-th level of the tree $T$, and  let $p_n \colon T_{n+1} \to T_n$ be the   predecessor map. We may assume that $T_1$ is infinite because the trees we obtain through the coding in \ref{th:1} have that property. Let   \[G_T = \varprojlim_{n \in \mathbb{N} } (L(T_{n+1}), \widehat p_{n+1}).\] 
\end{definition}
 We view  $G_T$ concretely as a closed subgroup of $\S$ according to Remark~\ref{rem:diagram to closed subgroup}.
 Needless to say, the map sending $T$ to $G_T$ is Borel.
 The rest of this section will establish the following, which is the second step towards proving the Main Result~\ref{th:main}.

 \begin{theorem}\label{the_hammer} Let $T$ and $  U $ be pruned trees on $\mathbb{N}$ such that $T_1$ and $U_1$ are infinite. Then 
 	$[T] \cong_u [U]$ if and only if $G_T$ is topologically isomorphic to~$G_U$.
 \end{theorem}


\begin{proof}  Let the maps $p_n\colon T_{n+1} \to T_n$ and $q_n\colon U_{n+1} \to U_n$ be  given by the predecessor maps  on the trees $T$ and $U$, respectively. Write     \[V_n = L(T_n) \qquad W_n = L(U_n).\] So we have 
		inverse systems $(V_n, \widehat p_n)_{n \in \mathbb N}$ and $  (W_n,  \widehat  q_n )_{n \in \mathbb N}$ over the category $\Gps$  determining  $G_T$ and $G_U$, respectively (see \ref{def:GT}). 
	
	 \smallskip 	 \noindent
 First suppose 	that $[T] \cong_u [U]$ via a uniform homeomorphism $\Phi$.  By \ref{lem:basic} there are increasing functions $\phi, \psi $ on~$\mathbb{N}$ and  onto  functions   \begin{center} $r_n \colon T_{\phi(n)}\to U_n$ and $s_n \colon U_{\psi(n)}\to T_n$ \end{center} so that for each $x\in[T]$ and $y \in [U]$ we have  
 \begin{center}$\Phi(x)\uhr n = r_n(x \uhr {\phi(n)})$ and $\Phi^{-1} (y)\uhr n = s_n(y \uhr {\psi(n)})$. \end{center}
 Let  $ f_n = \widehat r_n $ and   $q_n = \widehat s_n$  as defined in Subsection~\ref{ss:LV}.  The   pre-morphisms
	  \begin{center}
$(\bar f , \phi ) \colon (V_n, \widehat p_n)_n  \to (W_n,  \widehat  q_n )_n$  and   $(\bar g , \psi ) \colon (W_n,  \widehat  q_n )_n  \to  (V_n, \widehat p_n)_n  $  
	  \end{center}   induce morphisms in $\mrm{Pro}_\omega(\Gps)$. 
	  For each $t \in T_{\phi(\psi(n))} $ and $u \in U_{\psi(\phi(n))} $,
	    $$(s_n\circ r_{\psi(n)})(t) \subseteq t  \text{ and } (r_{n}\circ s_{\phi(n)})(u) \subseteq u.$$ 
	    
\smallskip \noindent	    
So $(V_n, \widehat p_n)_n$ and $  (W_n,  \widehat  q_n )_n$ are isomorphic in the category $\mrm{Pro}_\omega(\Gps)$ (to see this, recall    that the binding maps $\widehat{p}_n$ and $\widehat{q}_n$ act as ``predecessor maps'', since so do the maps $p_n$ and $q_n$). Hence $G_T$ and $G_U$ are isomorphic by~\ref{preserving_iso_progroups}. 
	    
\smallskip \noindent	    
 Conversely, suppose  that $G_T$ and $G_U$ are isomorphic. By  \ref{preserving_iso_progroups}   there are pre-morphisms  $(\ol f, \phi) \colon (V_n, \widehat p_n)_n  
 \to (W_n,  \widehat  q_n )_n $ and $(\ol g, \psi) \colon (W_n, \widehat q_n)_n  
 \to (V_n,  \widehat  p_n )_n $ that induce an isomorphism $(V_n, \widehat p_n)_n  
 \to (W_n,  \widehat  q_n )_n $ and its inverse, respectively,  in $\mrm{Pro}_\omega(\Gps)$. By virtue of being pre-morphisms, the following diagrams are commutative for each $k> n$:  
 \begin{center}
 	\begin{tikzcd}
 		V_{\phi(k)} \arrow[r, "f_k"] \arrow[d, "\widehat{p}_{\phi(k), \phi(n)}"'] & W_k \arrow[d, "\widehat{q}_{k,n}"] \\
 		V_{\phi(n)} \arrow[r, "f_n"] & W_n
 	\end{tikzcd}
 	\qquad
 	\begin{tikzcd}
 		W_{\psi(k)} \arrow[r, "g_k"] \arrow[d, "\widehat{q}_{\psi(k), \psi(n)}"'] 
 		& V_k \arrow[d, "\widehat{p}_{k,n}"] \\
 		W_{\psi(n)} \arrow[r, "g_n"] 
 		& V_n
 	\end{tikzcd}

 \end{center}
 Because the pre-morphisms induce  inverses of each other, we have 
 \begin{equation}\label{eqn:nice}\tag{$\star$} \widehat p _{\phi(\psi(n)),n} = g_n \circ f_{\psi(n)}\qquad   \widehat q _{\psi(\phi(n)),n} = f_n \circ g_{\phi(n)}. \end{equation}
 
 
 Let $r_n \colon \subseteq T_{\phi(n)} \to U_n$ be defined by \begin{center}
 	$r_n(t)= u$ if $f_n(t) $ is conjugate in $W_n$  to   $u$. 
 \end{center} (Note here that $f_n(t)^2 = e$, so $f_n(t) $ is conjugate to a unique generator, or it equals~$e$. The second case could \textit{prima facie}  happen,  so $r_n$ could be a partial map, as indicated in the notation.)    Similarly, let $s_n \colon \subseteq U_{\psi(n)} \to T_n$ be defined by \begin{center}
 	$s_n(u)= t$ if $g_n(u) $ is conjugate in $V_n$  to   $t$. 
 \end{center} 
 
 
 We claim that $s_n \circ r_{\psi(n)}= p_{\phi(\psi(n)),n}$ for each $n$.
 To verify this, let $t \in T_{\phi(\psi(n))}$. By (\ref{eqn:nice}) we have that $g_n(f_{\psi(n)}(t)) \neq e$. So
 $$f_{\psi(n)}(t) \sim_{W_{\psi(n)}}  r_{\psi(n)}(t) \in U_{\psi(n)}$$
 where $\sim_L$ denotes conjugacy in a group $L$. 
 So	    by (\ref{eqn:nice})  again
 \[\widehat p _{\phi(\psi(n)),n} (t)= g_n(f_{\psi(n)}(t)) \sim_{V_n} g_n( r_{\psi(n)}(t) ) \sim_{V_n} s_n (r_{\psi(n)}(t)).\]
 So $p _{\phi(\psi(n)),n} = s_n \circ r_{\psi(n)} $. Similarly, $q_{\psi(\phi(n)),n} = r_n \circ s_{\phi(n)}$. 
 Since the compositions are onto and total, each  map $r_n$ and $s_n$ is onto. Thus the maps $r_n$ and $s_n$ are also total (composing a partial after an onto map yields a partial map).  
 
 Corresponding to the diagrams in EpiGrp above,  we have commutative diagrams of onto maps between sets
 \begin{center}
 	\begin{tikzcd}
 		T_{\phi(k)} \arrow[r, "r_k"] \arrow[d, "p_{\phi(k), \phi(n)}"'] & U_k \arrow[d, "q_{k,n}"] \\
 		T_{\phi(n)} \arrow[r, "r_n"] & U_n
 	\end{tikzcd}
 	\qquad
 	\begin{tikzcd}
 		U_{\psi(k)} \arrow[r, "s_k"] \arrow[d, "q_{\psi(k), \psi(n)}"'] & T_k \arrow[d, "q_{k,n}"] \\
 		U_{\psi(n)} \arrow[r, "s_n"] & T_n
 	\end{tikzcd}
 \end{center} Hence  $[T] \cong_u [U]$ by \ref{lem:basic},  as required.

 \end{proof}
 
 \section{Strengthening to 2-nilpotent groups of exponent $p$}
Strengthening the previous result, this section establishes  a Borel reduction from uniform homeomorphism on Polish ultrametric spaces to isomorphism of nilpotent-2 procountable groups of exponent  $p$, where $p$ is a fixed  odd prime. We only need to modify the second step.

\subsection{The 2-nilpotent groups $H(X)$}
Let $N$ be the free nilpotent  group of class $2$ and exponent~$p$ on free generators $\{x_i : i \in \mathbb{N}\}$. For distinct $r ,  s \in \mathbb N$ we write $$ x_{r,s} = [x_r, x_s].$$ As noted in Mekler~\cite{Mekler1985}, the centre $Z(N)$ of $N$ is an elementary abelian $p$-group (and hence an  $\mathbb{F}_p$ vector space) with basis $x_{r,s}$ for $r< s$. Given an undirected irreflexive  graph~$A$ with vertex set $\NN$, Mekler   sets 
\begin{equation*} \tag{$\diamond$} \label{eqn:Mekler} G(A) = N/ \langle x_{r,s} \colon \, r A s \rangle. \end{equation*}
The centre  $Z= Z(G(A))$ is an elementary abelian $p$-group  with basis the $x_{r,s}$ such that  $\lnot r A s$.  Also $G(A)/Z $ is 
an abelian group of exponent $p$ freely generated by the $Z x_i$ (recall that in a $2$-nilpotent group the commutator subgroup is contained in the center).  
Mekler  obtained a normal form for elements of  $G(A)$.
\begin{lemma} \label{lem: NF G}    {\rm (i) Every element $c$ of $Z$ can be uniquely written  in the form 
		$$ \prod_{(r, s ) \in E} x_{r,s}^{\beta_{r,s}}$$
		where    $E \subseteq \omega \times \NN$ is a finite set of pairs $(r, s )$ with $r<s$ and  $\lnot rAs$, and $0 < \beta_{r,s}  < p$. 
		
		\noindent (ii) Every element of $G(A)$ can be uniquely written  in the form 
		$c \cdot v$  where  $c \in Z$, and $v=  \prod_{i \in C} x_i^{\alpha_i}$, for $C \subseteq \NN$   finite and $0 < \alpha_i < p$.  (The product  $\prod_{i \in C} x_i^{\alpha_i}$ is interpreted    along   the indices  in ascending order.)}
	
\end{lemma}

The following fact already noted in \cite{Mekler1985} is evident.
\begin{fact} \label{lem:com} Let $C, D \sub \NN$.  The following holds in $ G(A)$.  \[[\pp \alpha r C,  \pp \beta s D] = \prod_{r \in C, \, s \in D, \,  r< s, \, \lnot Ars} x_{r,s}^{\alpha_r \beta_s- \alpha_s \beta_r}\]\end{fact}  

We will only use Mekler's construction for a very simple infinite graph: the one such that each connected component has size 2. 

\begin{notation}\label{L_group_notation} For $n \in \mathbb N$ write $n^* = \lfloor n/2 \rfloor$.  Let $A = \{ \langle r, s \rangle \colon \, r \neq s \land r^* = s^*\}$. Let $L = G(A)$.  For $u \in L$ we write  $\underline u = uZ(L)$, which is an element of the group $L/Z(L)$.
\end{notation} 

The idea is that $G(A)$ has the minimum commutativity: any pair of elements that commute nontrivially have to generate the same subgroup as some pair  $x_i, x_{i+1}$ for even $i$, in the central quotient. 
\begin{lemma} \label{lem:commut} {\rm Suppose that $c,d \in L$ are such that $[c,d] =e$. 
		
		\noindent Either there is an even $i \in \mathbb N$ such that $\langle \underline c, \underline d \rangle = \langle \underline x_i, {\underline x}_{i+1} \rangle$, or  the rank of $\langle \underline c, \underline d \rangle$ as a subgroup of $L/Z(L)$ is at most $1$. }
\end{lemma} 

\begin{proof}
	We may assume that $c = \pp \alpha r C$ and $d = \pp \alpha s D$ as in \ref{lem: NF G}(ii). If $r \in \NN\setminus   C$   we let $\alpha_r=0$, and if $s \in \NN\setminus   D$ we let $\beta_s=0$. Suppose there is \emph{no}  even $i \in \mathbb N$ such that $\langle \underline c, \underline d \rangle = \langle \underline x_i, {\underline x}_{i+1} \rangle$. We show that $\underline c = \underline d ^\lambda$ for some $\lambda \in \mathbb N$. 
	
	\smallskip \noindent By assumption  there is no $k \in \NN$ such that both $C$ and $D$ are contained in \mbox{$\{2k,2k+1\}$}.  Note that neither $C$ nor $D$ can be a singleton: if (say) $C =\{r\}$, 
	then there is $s \in D$ such that $r^* \neq s^*$. Then $\alpha_r \beta_s \neq 0$ while $\alpha _s \beta_r=0$, which contradicts the assumption that $[c,d] \in Z(L)$  because of \ref{lem:com}.

	\smallskip \noindent Next  we claim that $C=D$.   Assume for contradiction that $C $ is not contained in $D$ (the case that $D$ is not contained in $C$ is symmetric).  Let $r \in C \setminus D$. Then there is $s \in D$ such that $r^* \neq s^*$ because $D$ is not a singleton. So $\alpha_r \beta_s\neq 0$ while $\alpha_s \beta_r =0$.   Thus $C=D$.

	\smallskip \noindent Now let $m = \min (C)$, and let $\lambda = \alpha_m/\beta_m$ as an element of $\mathbb F_p$. If $k \in C$ and $m ^* < k^*$,   then $\alpha_k/\beta_k = \lambda$. Some such $k$ exists. So if  $m$ is even  and $m+1 \in C$,  then we also have $ \alpha_{m+1}/\beta_{m+1} = \lambda$. 
	
	\smallskip \noindent We conclude that $\alpha_r = \beta_r \lambda $ holds in $\mathbb F_p$ for each $r \in C$. This shows that  $\underline c = \underline d ^\lambda$, and hence that the rank of $\langle \underline c, \underline d \rangle$  is at most~$1$.
\end{proof}

\begin{definition} \label{def:LX} Given a countably  infinite set $X$, let  $N(X)$  be the free 2-nilpotent group   of exponent~$p$   on distinct free generators   $\{a_v, b_v \colon v \in X\}$.  Let 
	\[H(X) = N(X) /\langle \{[a_v,b_v]  \colon v \in X\} \rangle. \]
	Given a set $X$, for $v\in X$  we will interpret    $a_v$ and $b_v$ in $H(X)$. Any surjection   $q \colon X \to Y$    induces a group epimorphism  $\widehat q \colon H(X) \to H(Y)$ given by mapping the equivalence class of $a_v $ to the equivalence  class  $a_{q(v)}$, and similarly for $b_v $ and $ b_{q(v)}$.  \end{definition} 
An arbitrary   bijection $h \colon X \to \mathbb N$ induces an isomorphism $H(X) \to L$ (for $L$ as in~\ref{L_group_notation}) by sending $a_v$ to the equivalence class of $x_{2h(v)}$ and $b_v$ to the equivalence class of $x_{2h(v) + 1}$ given by the presentation in (\ref{eqn:Mekler}). So we can apply Lemma~\ref{lem:commut} to $H(X)$ as well. 
This lemma can now be used to turn certain group epimorphism $H(X) \to H(Y)$ into surjections  $X \to Y$ in a functorial manner. This will be needed  to eventually verify that  the reduction for the second step (\ref{def:GT} below) works.

If $r \colon G \to H$ is an epimorphism of groups, we denote by $\ul r$ the induced epimorphism $G/Z(G) \to H/Z(H)$. 

\begin{lemma} \label{lem: silly}
	{\rm Let $\+ L$ be the category of groups of the form $H(X)$ with morphisms the epimomorphisms $g \colon H(X) \to H(Y)$ such that  for each $v \in X$,  the rank of $\langle \underline {g(a_v)}, \underline {g(b_v)} \rangle$ as a subgroup of $H(Y)/Z(H(Y))$ equals $2$. A functor $F $ from $ \+ L $ to the category of countable sets with mappings between them  is declared by sending $H(X)$ to $X$, and sending a morphism $g: H(X) \to H(Y)$ to the map $p= F(g) \colon X \to Y$ such that    \[\langle \underline {g(a_v)}, \underline {g(b_v)} \rangle = \langle \underline {a_{p(v)}}, \underline {b_{p(v)}}\rangle. \]  Furthermore, if $q \colon X \to Y$ is a surjection, then $\widehat  q$ is a morphism of $\+ L$, and $F(\widehat q) = q$.} \end{lemma}

\begin{proof}
	To show that $F$ is well-defined for a morphism $g$ of $\+ L$ as above, 	it suffices to apply   \ref{lem:commut} to  $L = H(Y)$, $c=g(a_v) , d= g(b_v)$, noting that $[c,d]= e$. Thus in $H(Y)/Z(H(Y))$ we have $ \langle  \ul c, \ul d \rangle = \langle \ul a_y, \ul  b_y \rangle$ for a unique $y \in Y$, and hence $p(v) = y$ is defined.  One verifies easily that  $F$ is a functor, and that $F(\widehat q)= q$. 
\end{proof}

\subsection{The  new reduction for the second step}
	Given a pruned tree $T$, let  the sets $T_n $ be as before, and let   \[H_T = \varprojlim_{n \in \mathbb{N} } (H(T_{n+1}), \widehat p_{n+1}).\] 
The map sending $T$ to $H_T$ is Borel.
 
\begin{theorem}\label{the_hammer nil} Let $T$ and $  U $ be pruned trees on $\mathbb{N}$ such that $T_1$ and $U_1$ are infinite. Then 
	$[T] \cong_u [U]$ if and only if $H_T$ is topologically isomorphic to~$H_U$.
\end{theorem}


\begin{proof}  As before, let the maps $p_n\colon T_{n+1} \to T_n$ and $q_n\colon U_{n+1} \to U_n$ be  given by the predecessor maps  on the trees $T$ and $U$, respectively. Now write     \[V_n = H(T_n) \qquad W_n = H(U_n).\] So we have 
	inverse systems $(V_n, \widehat p_n)_{n \in \mathbb N}$ and $  (W_n,  \widehat  q_n )_{n \in \mathbb N}$ over the category $\Gps$  determining  $H_T$ and $H_U$, respectively (see \ref{def:LX}). 
	
	\smallskip 	 \noindent
If  $[T] \cong_u [U]$ via a uniform homeomorphism $\Phi$, then $H_T \cong H_U$ as before.
Now  suppose  that $H_T$ and $H_U$ are isomorphic. By  \ref{preserving_iso_progroups}   there are pre-morphisms  $(\ol f, \phi) \colon (V_n, \widehat p_n)_n  
	\to (W_n,  \widehat  q_n )_n $ and $(\ol g, \psi) \colon (W_n, \widehat q_n)_n  
	\to (V_n,  \widehat  p_n )_n $ that induce an isomorphism $(V_n, \widehat p_n)_n  
	\to (W_n,  \widehat  q_n )_n $ and its inverse, respectively,  in $\mrm{Pro}_\omega(\Gps)$. 
	Because the pre-morphisms induce  inverses of each other, we have 
	\begin{equation}\label{eqn:nice}\tag{$\star$} \widehat p _{\phi(\psi(n)),n} = g_n \circ f_{\psi(n)}\qquad   \widehat q _{\psi(\phi(n)),n} = f_n \circ g_{\phi(n)}. \end{equation}
	Using the notation of \ref{lem: silly}, the maps on the left hand sides of the equation are morphisms of the category $\+ L$. We claim that  $g_n$ and $f_n$ are also morphisms of $\+ L$ for each $n$.  We verify it  for $g_n$, the case of $f_n$ being symmetric. We use  the   left equation in (\ref{eqn:nice}). Write $p$ for $p _{\phi(\psi(n)),n}$. For each $v \in  T_{\phi(\psi(n)} $, by definition we have $\widehat p (\langle a_v, b_v \rangle) = \langle  a_{p(v)},  b_{p(v)} \rangle$. Thus  the   rank of $\langle \ul a_v, \ul b_v \rangle$ cannot be reduced by  applying    the homomorphism  induced by $f_{\psi(n)}$ on  the central quotients. So $f_{\psi(n)}$ is in $\+ L$ by \ref{lem: silly}, and hence $g_n$ is in $\+ L$ by a  similar argument. This shows the claim.
	
	\smallskip \noindent
	Since $(\ol f, \phi)$ and $(\ol g, \psi)$  are pre-morphisms, the following diagrams commute  for each $k> n$:  
	\begin{center}
		\begin{tikzcd}
			V_{\phi(k)} \arrow[r, "f_k"] \arrow[d, "\widehat{p}_{\phi(k), \phi(n)}"'] & W_k \arrow[d, "\widehat{q}_{k,n}"] \\
			V_{\phi(n)} \arrow[r, "f_n"] & W_n
		\end{tikzcd}
		\qquad
		\begin{tikzcd}
			W_{\psi(k)} \arrow[r, "g_k"] \arrow[d, "\widehat{q}_{\psi(k), \psi(n)}"'] 
			& V_k \arrow[d, "\widehat{p}_{k,n}"] \\
			W_{\psi(n)} \arrow[r, "g_n"] 
			& V_n
		\end{tikzcd}

	\end{center}
	Let 
	$r_n \colon   T_{\phi(n)} \to U_n$ and  $s_n \colon   U_{\psi(n)} \to T_n$ be defined by $r_n = F(f_n)$ and $s_n = F(g_n)$ where $F$ is the functor from Lemma~\ref{lem: silly}. 
	By that lemma   we obtain   commutative diagrams of maps  between sets
	\begin{center}
		\begin{tikzcd}
			T_{\phi(k)} \arrow[r, "r_k"] \arrow[d, "p_{\phi(k), \phi(n)}"'] & U_k \arrow[d, "q_{k,n}"] \\
			T_{\phi(n)} \arrow[r, "r_n"] & U_n
		\end{tikzcd}
		\qquad
		\begin{tikzcd}
			U_{\psi(k)} \arrow[r, "s_k"] \arrow[d, "q_{\psi(k), \psi(n)}"'] & T_k \arrow[d, "q_{k,n}"] \\
			U_{\psi(n)} \arrow[r, "s_n"] & T_n
		\end{tikzcd}
	\end{center} Let $\Phi \colon [T] \to [U]$ and $\Psi \colon [U] \to [T]$ be defined by 
	\[ \Phi(x) = \bigcup_n r_n( x \upharpoonright \phi(n)) \qquad  \Psi(y) = \bigcup_n s_n( y \upharpoonright \psi(n)). \] Then $\Phi$ and $\Psi$ are uniformly continuous  by \ref{lem:basic}.   By (\ref{eqn:nice}) and \ref{lem: silly} we have $p_{\phi(\psi(n)), n}= s_n \circ r_{\psi(n)}$ and $q_{\psi(\phi(n)), n}= r_n \circ s_{\phi(n)}$ for each $n$. Hence  $\Psi \circ \Phi$ is the identity on $[T]$, and $\Phi \circ \Psi$ is the identity on $[U]$. So $[T] \cong_u [U]$  as required.
\end{proof}

\end{document}